\documentclass[10pt,a4paper,reqno]{amsart}
\usepackage{amsthm}
\usepackage{amsmath}
\usepackage{amssymb}
\usepackage{graphicx,color}

\usepackage[font=small]{caption}
\usepackage{subcaption}

\usepackage{multirow}
\usepackage[shortlabels]{enumitem}

\usepackage{cancel}

\makeatletter
\theoremstyle{plain}
\newtheorem{thm}{Theorem}
  \theoremstyle{definition}
  \newtheorem*{thm*}{Theorem}
  
  \theoremstyle{remark}
  
  \theoremstyle{plain}
  \newtheorem{prop}[thm]{Proposition}
  \theoremstyle{plain}
  \newtheorem{lem}[thm]{Lemma}
  \theoremstyle{plain}
  
 \theoremstyle{definition}
  
  \theoremstyle{remark}
  \newtheorem*{rem*}{Remark}

  \theoremstyle{definition}

\usepackage{amsfonts}
\usepackage{mathrsfs}

\newtheorem*{question*}{\it{QUESTION}}

\theoremstyle{plain}

\newtheorem*{oq*}{Open Question}

\newcommand{\N}{\mathbb{N}}
\newcommand{\R}{{\mathbb{R}}}
\newcommand{\C}{{\mathbb{C}}}

\newcommand{\Z}{{\mathbb{Z}}}
\newcommand{\dd}{{\rm d}}
\newcommand{\ii}{{\rm i}}

\newcommand{\spn}{\mathop\mathrm{span}\nolimits} % \span...makes nonfunctional /align

\newcommand{\Ran}{\mathop\mathrm{Ran}\nolimits}

\renewcommand{\Re}{\mathop\mathrm{Re}\nolimits}
\renewcommand{\Im}{\mathop\mathrm{Im}\nolimits}

\newcommand{\dist}{\mathop\mathrm{dist}\nolimits}

\makeatother

\begin{document}

\title[]{On Lieb--Thirring inequalities for one-dimensional non-self-adjoint Jacobi and Schr{\" o}dinger operators}

\author{Sabine B{\"o}gli}
\address[Sabine B{\"o}gli]{
   Department of Mathematical Sciences, Durham University, Lower Mountjoy, Stockton Road, Durham DH1 3LE, UK
    }
\email{sabine.boegli@durham.ac.uk}

\author{Franti\v sek \v Stampach}
\address[Franti{\v s}ek {\v S}tampach]{
	Department of Mathematics, Faculty of Nuclear Sciences and Physical Engineering, Czech Technical University in Prague, Trojanova~13, 12000 Praha~2, Czech Republic
	}	
\email{stampfra@fjfi.cvut.cz}



\date{\today}

\begin{abstract}
We study to what extent Lieb--Thirring inequalities are extendable from self-adjoint to general (possibly non-self-adjoint) Jacobi and Schr{\" o}dinger operators. Namely, we prove the conjecture of Hansmann and Katriel from~\cite{han-kat_caot11} and answer another open question raised therein. The results are obtained by means of asymptotic analysis of eigenvalues of discrete Schr{\" o}dinger operators with rectangular barrier  potential and complex coupling. Applying the ideas in the continuous setting, we also solve a similar open problem for one-dimensional Schr{\" o}dinger operators with complex-valued potentials published by Demuth, Hansmann, and Katriel in~\cite{dem-han-kat_ieop13}.
\end{abstract}

\maketitle

\section{Introduction}

Lieb--Thirring inequalities have attracted the attention of the mathematical community since their appearance in the work of Lieb and Thirring~\cite{lie-thi_prl75,lie-thi_91} on the stability of matter, where they were carried out in the context of self-adjoint Schr{\" o}dinger operators. Later developments gave rise to a huge number of works devoted primarily to Lieb--Thirring inequalities for Schr{\" o}dinger operators but also other operator families. For some references concerning Lieb--Thirring inequalities for Schr{\" o}dinger and Jacobi operators, we mention at least \cite{dem-han-kat_jfa09,dem-han-kat_ieop13,fra_tams18,fra-etal_06,fra-sim-wei_cmp08,hun-lie-tho_atmp98, hun-sim_jat02, wei_cmp96}.

Within the last decade, a great interest developed for generalizations of the classical Lieb--Thirring inequalities, which were originally derived for self-adjoint operators only, to non-self-adjoint operator families. Still several naturally formulated questions have remained open. Here we particularly refer to the open problems concerning non-self-adjoint Jacobi and Schr{\" o}dinger operators that were published in~\cite{han-kat_caot11} and~\cite{dem-han-kat_ieop13} and which are discussed in this article in more detail. As far as the existing results on Lieb--Thirring inequalities for non-self-adjoint Jacobi operators are concerned, the reader may consult the papers \cite{bor-gol-kup_blms09, chri-zin_lmp17, gol-kup_lmp07,han_lmp11,han-kat_caot11}.

\subsection{State of the art - Jacobi operators}

Let $J$ be the Jacobi operator acting on~$\ell^{2}(\Z)$ defined by its action on vectors of the standard basis $\{e_{n}\}_{n\in\Z}$ of $\ell^{2}(\Z)$ by
\[
 Je_{n}=%a_{n-1}e_{n-1}+b_{n}e_{n}+c_{n}e_{n+1},
a_ne_{n+1}+b_{n}e_{n}+c_{n-1}e_{n-1}, \quad n\in\Z,
\]
where $\{a_{n}\}_{n\in\Z}$, $\{b_{n}\}_{n\in\Z}$, and $\{c_{n}\}_{n\in\Z}$ are given bounded complex sequences. Then $J$ is a bounded operator and can be identified with the doubly-infinite complex Jacobi matrix
\[
J=\begin{pmatrix}
	\ddots & \ddots & \ddots & \\
	 & a_{-1} & b_{0} & c_{0} & \\
	 & & a_{0} & b_{1} & c_{1} & \\
 	 & & & a_{1} & b_{2} & c_{2} & \\
     & & & & \ddots & \ddots & \ddots
  \end{pmatrix}.
\]
%
%This article is motivated by recent developments on Lieb--Thirring-type inequalities for Jacobi matrices. Namely, their extensions from the well-known self-adjoint case, when $a_{n}=c_{n}>0$ and $b_{n}\in\R$, to the general complex case is of our primary interest. A~formulation and proof of such an extension is the subject of recent paper~\cite{han-kat_caot11}, where certain natural questions remained open, however. This contribution provides answers to these questions. 

We follow~\cite{han-kat_caot11} and use the notation
\[
 d_{n}:=\max\{|a_{n-1}-1|,|a_{n}-1|,|b_{n}|,|c_{n-1}-1|,|c_{n}-1|\}, \quad n\in\Z.
\]
If $\lim_{n\to\pm\infty}d_{n}=0$, $J$ is a compact perturbation of the free Jacobi operator $J_{0}$ defined by 
\[
 J_{0}e_{n}=e_{n-1}+e_{n+1}, \quad n\in\Z.
\]
In this case, it is well known that the essential spectrum is $\sigma_{ess}(J)=\sigma_{ess}(J_0)=[-2,2]$ and
\[
\sigma(J)=[-2,2]\cup\sigma_{d}(J).
\]
The discrete spectrum $\sigma_{d}(J)\subset\C\setminus[-2,2]$  is an at most countable set of eigenvalues of~$J$ with all possible accumulation points contained in $[-2,2]$.

Lieb--Thirring inequalities for self-adjoint Jacobi operators, i.e, for the case when $a_{n}=c_{n}>0$ and $b_{n}\in\R$ are due to Hundertmark and Simon~\cite{hun-sim_jat02} and can be formulated as follows: If $d\in\ell^{p}(\Z)$ for some $p\geq1$, then
\begin{equation}
\sum_{\lambda\in\sigma_{d}(J)\cap(-\infty,-2)}|\lambda+2|^{p-1/2}+\sum_{\lambda\in\sigma_{d}(J)\cap(2,\infty)}|\lambda-2|^{p-1/2}\leq C_{p}\|d\|_{\ell^{p}}^{p},
\label{eq:hun-sim_ineq1}
\end{equation}
where $C_{p}$ is an explicit constant that depends on $p$ but is independent of~$J$. Such constants are meant generically and can vary while in the following the notation remains the same.

When trying to find a convenient form for an extension of inequality~\eqref{eq:hun-sim_ineq1} to non-self-adjoint Jacobi operators, it seems natural to reformulate~\eqref{eq:hun-sim_ineq1} in terms of the distance between the eigenvalue $\lambda$ and the essential spectrum $[-2,2]$ as
\begin{equation}
 \sum_{\lambda\in\sigma_{d}(J)}\left(\dist(\lambda,[-2,2])\right)^{p-1/2}\leq C_{p}\|d\|_{\ell^{p}}^{p}.
\label{eq:hun-sim_ineq2}
\end{equation}
In~\cite{han-kat_caot11}, Hansmann and Katriel conjectured that the inequality~\eqref{eq:hun-sim_ineq2} is no longer true when the assumption on self-adjointness of $J$ is dropped. Our first main result (Theorem~\ref{thm:first})  proves 
the conjecture. In fact, we show that~\eqref{eq:hun-sim_ineq2} does not hold even when restricted to non-self-adjoint discrete Schr{\" o}dinger operators, i.e, Jacobi operators $J$ with $a_{n}=c_{n}=1$, for all $n\in\Z$.

Another form of the inequality~\eqref{eq:hun-sim_ineq1} that is admissible for an extension to the non-self-adjoint case can be based on the observation that
\[
\frac{\dist\left(\lambda,[-2,2]\right)^{p}}{|\lambda^{2}-4|^{1/2}}\leq\frac{1}{2}\begin{cases}
														  |\lambda-2|^{p-1/2},& \quad \mbox{ if } \lambda\in (2,\infty),\\
														  |\lambda+2|^{p-1/2},& \quad \mbox{ if } \lambda\in (-\infty,-2).
														  \end{cases}
\]
Then~\eqref{eq:hun-sim_ineq1} implies, for the self-adjoint case,
\begin{equation}
\sum_{\lambda\in\sigma_{d}(J)}\frac{\left(\dist(\lambda,[-2,2])\right)^{p}}{|\lambda^{2}-4|^{1/2}}\leq C_{p}\|d\|_{\ell^{p}}^{p}.
\label{eq:hans-kat_tau0}
\end{equation}
Note that its generalization to the non-real case would be a weaker version than~\eqref{eq:hun-sim_ineq2} since $\dist(\lambda,[-2,2])\leq \frac 1 2 |\lambda^2-4|$ for all $\lambda\in\C$.
The estimate~\eqref{eq:hans-kat_tau0} is very close to the even weaker version that was proven for general, possibly non-self-adjoint, Jacobi operators in~\cite[Thm.~1]{han-kat_caot11}.

\begin{thm}[Hansmann--Katriel] \label{thm:hans-kat}
 Suppose $\tau\in(0,1)$ and $d\in\ell^{p}(\Z)$ with $p\geq1$. Then
 \begin{equation}
 \sum_{\lambda\in\sigma_{d}(J)}\frac{\left(\dist(\lambda,[-2,2])\right)^{p+\tau}}{|\lambda^{2}-4|^{1/2}}\leq C_{p,\tau}\| d\|_{\ell^{p}}^{p}, \quad \mbox{ if } p>1,
 \label{eq:han-kat_ineq_p}
 \end{equation}
 and
 \begin{equation}
 \sum_{\lambda\in\sigma_{d}(J)}\frac{\left(\dist(\lambda,[-2,2])\right)^{1+\tau}}{|\lambda^{2}-4|^{1/2+\tau/4}}\leq   C_{\tau}\| d\|_{\ell^{1}}, \quad \mbox{ if } p=1.
 \label{eq:han-kat_ineq_1}
 \end{equation}
\end{thm}

The inequalities~\eqref{eq:han-kat_ineq_p} and~\eqref{eq:han-kat_ineq_1} are slightly weaker than~\eqref{eq:hans-kat_tau0} due to the presence of the positive parameter $\tau$. The proof presented in~\cite{han-kat_caot11} elaborates on a previous result due to Borichev et al.~\cite{bor-gol-kup_blms09}, where the parameter $\tau$ enters and its positivity is required by the chosen approach. However, it remained an open question whether~\eqref{eq:han-kat_ineq_p} and~\eqref{eq:han-kat_ineq_1} could hold for $\tau=0$, which would imply~\eqref{eq:hans-kat_tau0}.
Our second main result (Theorem~\ref{thm:second})  answers this question to the negative, i.e., inequality~\eqref{eq:hans-kat_tau0} does not extend to non-self-adjoint Jacobi operators. In fact, it is not even true for non-self-adjoint discrete Schr{\" o}dinger operators. This means that the positivity of $\tau$ is not just a requirement dictated by the chosen approach in~\cite{han-kat_caot11} but it is essential. Consequently, Theorem~\ref{thm:hans-kat} is sharp in this sense. Recently, Theorem~\ref{thm:hans-kat} was generalized to non-self-adjoint perturbations of finite gap Jacobi matrices by Christiansen and Zinchenko in~\cite{chri-zin_lmp17}.

Although the answers to both questions raised in~\cite{han-kat_caot11} are negative,  they help to better understand the boundaries between the self-adjoint and general setting for Lieb--Thirring-type inequalities. The strategy to obtain the answers is based on a convenient choice of a concrete family of Jacobi operators from the considered class. We study the discrete Schr{\" o}dinger operator with rectangular barrier  potential and complex coupling. The properties of this particular operator can be of independent interest. For our goals, it is essential that the eigenvalue problem can be transformed into a study of solutions of relatively simple algebraic equations. These results are worked out in Section~\ref{sec:jacobi}.

\subsection{State of the art - Schr{\" o}dinger operators}

A similar open problem, this time for Schr{\" o}dinger operators with complex-valued potentials, was published in~\cite{dem-han-kat_ieop13}. Recall that the classical Lieb--Thirring inequality for a Schr{\" o}dinger operator $H=-\Delta+V$ in $L^{2}(\R^{d})$ reads
\begin{equation}
\sum_{\lambda\in\sigma_{d}(H)}|\lambda|^{p-d/2}\leq C_{p,d}\|V\|_{L^{p}}^{p},
\label{eq:lieb-thirring_ineq_sa}
\end{equation}
provided that $V$ is a real-valued function from $L^{p}(\R^{d})$, where the range for $p$ depends on the dimension~$d$ as follows:
\begin{align}
p\geq 1,& \quad \mbox{ if } d=1,\nonumber\\
p> 1,& \quad \mbox{ if } d=2,\label{eq:p_d_rel}\\
p\geq\tfrac{d}{2},& \quad \mbox{ if } d\geq 3.\nonumber
\end{align}

Inequality~\eqref{eq:lieb-thirring_ineq_sa} cannot be true for complex-valued $V\in L^{p}(\R^{d})$ with $p>d$ since, in this case, $\sigma_{d}(H)$ can have accumulation points anywhere in $\sigma_{ess}(H)=[0,\infty)$, see~\cite{bog_cmp17}. However, if $|\lambda|^{p}$ is replaced by $(\dist(\lambda,[0,\infty)))^{p}$ in~\eqref{eq:lieb-thirring_ineq_sa}, we arrive at the inequality
\begin{equation}
\sum_{\lambda\in\sigma_{d}(H)}\frac{\left(\dist(\lambda,[0,\infty))\right)^{p}}{|\lambda|^{d/2}}\leq C_{p,d}\|V\|_{L^{p}}^{p},
\label{eq:lieb-thirring_ineq_dist_form}
\end{equation}
which seems to be a reasonable candidate for the Lieb--Thirring inequality extended to complex-valued potentials. This brings us to the following open problem formulated in~\cite{dem-han-kat_ieop13}.

\begin{oq*}[Demuth--Hansmann--Katriel]
 Assuming~\eqref{eq:p_d_rel}, is inequality~\eqref{eq:lieb-thirring_ineq_dist_form} true for all $V\in L^{p}(\R^{d})$? Prove it or construct a counter-example.
\end{oq*}

In Theorem~\ref{thm:third} we partly answer the question by showing it is again negative for $d=1$, see the construction of a concrete counter-example in Section~\ref{sec:schrodinger}. The approach is similar as the one used in the discrete case of Jacobi matrices. Note that, for $d=1$, the inequality~\eqref{eq:lieb-thirring_ineq_dist_form} can be viewed as a~continuous analogue of the inequality~\eqref{eq:hans-kat_tau0}.  The problem remains open, however, in higher dimensions $d\geq2$.

\section{Jacobi operators}\label{sec:jacobi}

For the sake of concreteness, we formulate two statements whose proofs follow from the analysis of properties of the discrete Schr{\" o}dinger operator with rectangular barrier  potential and complex coupling studied below. To distinguish, in notation, the restriction of the class of general Jacobi operators $J$ with $d\in\ell^{p}(\Z)$ to the set of discrete Schr{\" o}dinger operators with complex potential $b\in\ell^{p}(\Z)$, we denote by $T=T(b)$ the operator determined by the equations
\[
Te_{n}:=e_{n-1}+b_{n}e_{n}+e_{n+1}, \quad n\in\Z.
\]

\begin{thm}\label{thm:first}
For any $p\geq0$ and $\omega<p$, one has
\[
\sup_{0\neq b\in\ell^{p}(\Z)}\frac{1}{\|b\|_{\ell^{p}}^{p}}\sum_{\lambda\in\sigma_{d}(T(b))}\left(\dist(\lambda,[-2,2])\right)^{\omega}=\infty.
\]
\end{thm}

In particular, for $\omega=p-1/2$, Theorem~\ref{thm:first} confirms the conjecture of Hansmann and Katriel. On the other hand, the inequality
\[
 \sum_{\lambda\in\sigma_{d}(J)}\left(\dist(\lambda,[-2,2])\right)^{p}\leq C_{p}\|d\|_{\ell^{p}}^{p}
\]
is known to hold for any Jacobi operator $J$, see~\cite[Thm.~4.2]{han_lmp11}. 
Hence, for $\omega\geq p\geq1$, the claim of Theorem~\ref{thm:first} is no longer true.
This shows the difference between the self-adjoint and general case for this kind of Lieb--Thirring inequalities for the exponent $\omega$ in the interval $[p-1/2,p)$.

The next statement concerns the possibility of extension of inequality~\eqref{eq:hans-kat_tau0} to the non-self-adjoint setting.

\begin{thm}\label{thm:second}
For any $p\geq1$ and $\sigma\geq1/2$, one has
\[
\sup_{0\neq b\in\ell^{p}(\Z)}\frac{1}{\|b\|_{\ell^{p}}^{p}}\sum_{\lambda\in\sigma_{d}(T(b))}\frac{\left(\dist(\lambda,[-2,2])\right)^{p}}{|\lambda^{2}-4|^{\sigma}}=\infty.
\]
\end{thm}

If we put $\sigma=1/2$, Theorem~\ref{thm:second} shows that~\eqref{eq:hans-kat_tau0} does not hold for general Jacobi operators. In other words, Theorem~\ref{thm:hans-kat} is no longer true when $\tau=0$.

\subsection{Discrete Schr{\" o}dinger operator with rectangular barrier  potential and complex coupling}

For $n\in\N$ and $\beta\in\C$, we consider the two-parameter family of discrete Schr{\"o}dinger operators $T=T_{\beta,n}$ determined by the potential 
\[
 b_{k}:=\begin{cases}
 		\beta,& \quad \mbox{ for } k\in\{1,2\dots,n\},\\
 		0,& \quad \mbox{ for } k\in\Z\setminus\{1,2\dots,n\}.
 		\end{cases}
\]
Alternatively, $T_{\beta,n}$ can be written in the form
\[
 T_{\beta,n}=J_{0}+\beta P_{n},
\]
where $J_{0}$ is the free Jacobi operator (or the discrete Laplacian) and $P_{n}$ the orthogonal projection onto $\spn\{e_{1},\dots,e_{n}\}$. The operator $T_{\beta,n}$ is a discrete analogue of the Schr{\" o}dinger operator with rectangular barrier  potential supported on the set $\{1,\dots,n\}$ and complex coupling parameter~$\beta$.

Our first goal is a general spectral analysis of $T_{\beta,n}$ which can be of independent interest. However, we restrict the coupling constant $\beta$ to purely imaginary which is sufficient for our later purpose. Without loss of generality, we can even assume $\beta=\ii h$ for $h>0$. The discrete spectrum of such an operator is located in the rectangular domain $[-2,2]+\ii(0,h]$. 

\begin{lem}\label{lem:rough_local_evls}
 Let $h>0$. If $\lambda\in\sigma_{d}(T_{\ii h,n})$, then
 \[
  -2\leq\Re \lambda\leq 2 \quad\mbox{ and }\quad 0<\Im\lambda\leq h,
 \]
 for all $n\geq2$.
\end{lem}

\begin{proof}
 The proof is based on the enclosure of the spectrum by the numerical range.
 Let $\lambda\in\sigma_{d}(T_{\ii h,n})$ and $\phi\in\ell^{2}(\Z)$ be a corresponding normalized eigenvector. Then
 \[
  |\Re\lambda|=|\langle\phi,(\Re T_{\ii h,n})\phi\rangle|=|\langle\phi,J_{0}\phi\rangle|\leq\|J_{0}\|=2.
 \]
 Similarly, one has
 \[
  \Im\lambda=\langle\phi,(\Im T_{\ii h,n})\phi\rangle=h\langle\phi,P_{n}\phi\rangle=h\|P_{n}\phi\|^{2},
 \]
 which readily implies $\Im\lambda\leq h$ and also
 \[
  \Im\lambda\geq h\left(|\phi_{1}|^{2}+|\phi_{2}|^{2}\right)>0
 \]
 because $n\geq2$. The last expression cannot vanish indeed, since if $\phi_{1}=\phi_{2}=0$, then  it follows from the eigenvalue equation $T_{\ii h,n}\phi=\lambda\phi$ that $\phi=0$, contradicting the assumption $\|\phi\|=1$.
\end{proof}

Next, we look at the eigenvalues of $T_{\beta,n}$ more closely. By the Birman--Schwinger principle, $\lambda\notin[-2,2]$ is an eigenvalue of $T_{\beta,n}$ if and only if $-1$ is an eigenvalue of the the operator $\beta P_{n}(J_{0}-\lambda)^{-1}P_{n}$ which has finite rank. This observation provides us with a characteristic equation for the discrete spectrum of $T_{\beta,n}$:
\[
 \lambda\in\sigma_{d}\left(T_{\beta,n}\right)\quad\Leftrightarrow\quad\det(1+\beta P_{n}(J_{0}-\lambda)^{-1}P_{n})=0.
\]

Recall that the Joukowsky conformal mapping $k\mapsto k+k^{-1}$ maps bijectively the punctured unit disk $\{k\in\C \mid 0<|k|<1\}$ onto $\C\setminus[-2,2]$. Writing $\lambda=k+k^{-1}$, for $0<|k|<1$, a standard computation shows
\[
 (J_{0}-\lambda)^{-1}=\frac{k}{k^{2}-1}Q(k),
\]
where $Q(k)$ is the Laurent operator with entries $\left(Q(k)\right)_{i,j}=k^{|j-i|}$, see, for example, \cite[Prop.~2.6]{hun-sim_jat02}. Let $Q_{n}(k)$ denote the finite section matrix obtained from $Q(k)$ by restricting the indices to~$\{1,\dots,n\}$, i.e., $Q_{n}(k)=P_{n}Q(k)P_{n}\upharpoonleft\Ran P_{n}$. Spectral properties of the matrix~$Q_{n}(k)$, sometimes called the Kac--Murdock--Szeg{\H o} matrix, are studied in~\cite{fik_laa18} for a general $k\in\C$. Particularly, the characteristic polynomial of $Q_{n}(k)$ is expressible in terms of the Chebyshev polynomials of the second kind $U_{n}$, see~\cite[Eq.~(2.4)]{fik_laa18}. Using these facts, we obtain the expression
\begin{align}
 \det(1+\beta P_{n}(J_{0}-\lambda)^{-1}P_{n})&=\det\left(1+\frac{k\beta}{k^{2}-1}Q_{n}(k)\right)\nonumber\\
 &=\frac{k^{n}}{1-k^{2}}\left[U_{n}\left(\xi\right)-2kU_{n-1}\left(\xi\right)+k^{2}U_{n-2}\left(\xi\right)\right]\!, \label{eq:char_func_chebys}
\end{align}
where
\[
 \xi=\frac{k+k^{-1}-\beta}{2}.
\]

Taking further into account the well known identity for Chebyshev polynomials
\[
 U_{n}\left(\frac{z+z^{-1}}{2}\right)=\frac{z^{n+1}-z^{-n-1}}{z-z^{-1}}, \quad n\in\N_{0},
\]
it is natural to introduce a new parameter $z$ by the equation
\begin{equation}
 \beta=k+k^{-1}-z-z^{-1}.
\label{eq:bet_z_k}
\end{equation}
Then, using~\eqref{eq:char_func_chebys}, one gets the explicit formula
\[
 \det\left(1+\frac{k\beta}{k^{2}-1}Q_{n}(k)\right)=\frac{k^{2n}}{1-k^{2}}\frac{\beta^{n}}{(z-k)^{n}(1-kz)^{n}}\frac{z^{2n}(z-k)^{2}-(1-kz)^{2}}{z^{2}-1}.
\]
Zeros of the determinant are solutions of the equation
\[
 z^{2n}(z-k)^{2}-(1-kz)^{2}=0,
\]
which, when solved for $k=k(z)$, yields
\begin{equation}
 k=\frac{z^{n+1}-1}{z^{n}-z} \quad\mbox{ or }\quad k=\frac{z^{n+1}+1}{z^{n}+z}.
\label{eq:k_rel_z_plusminus}
\end{equation}
Inserting the above expressions for $k$ back into~\eqref{eq:bet_z_k}, we arrive at two polynomial equations
\begin{equation}
 \beta\left(z^{n+1}-1\right)\left(z^{n-1}-1\right)-z^{n-2}\left(z^{2}-1\right)^{2}=0
\label{eq:beta_z_minus}
\end{equation}
and
\begin{equation}
 \beta\left(z^{n+1}+1\right)\left(z^{n-1}+1\right)+z^{n-2}\left(z^{2}-1\right)^{2}=0,
\label{eq:beta_z_plus}
\end{equation}
for $n\geq2$. The solutions of equations~\eqref{eq:beta_z_minus} or~\eqref{eq:beta_z_plus} have the following properties whose verification is straightforward.

\begin{lem}\label{lem:sol_z_basic_prop}
 The solutions of \eqref{eq:beta_z_minus} or~\eqref{eq:beta_z_plus} are invariant under the transformation $z\leftrightarrow z^{-1}$. Suppose further that $\beta=\ii h$ with $h>0$. Then the only solutions of~\eqref{eq:beta_z_minus} located on the unit circle are two double roots $z=\pm1$ if $n$ is odd, and one double root $z=1$ if $n$ is even. Similarly, the only solution of~\eqref{eq:beta_z_plus} located on the unit circle is one double root $z=-1$ if $n$ is even, and no solution if $n$ is odd. In addition, if $n$ is odd, then the solutions of \eqref{eq:beta_z_minus} or~\eqref{eq:beta_z_plus} are invariant under the transformation $z\leftrightarrow -\overline{z}$ (symmetry w.r.t.\ the imaginary axis) and, if $n$ is even, then $z$ is a solution of~\eqref{eq:beta_z_minus} if and only if $-\overline{z}$ is a solution of~\eqref{eq:beta_z_plus}.
\end{lem}

Lemma~\ref{lem:sol_z_basic_prop} allows us to restrict the analysis of the solutions of~\eqref{eq:beta_z_minus} and~\eqref{eq:beta_z_plus} to the unit disk $|z|<1$. 
Since the polynomials in~\eqref{eq:beta_z_minus} and~\eqref{eq:beta_z_plus} are of degree $2n$, 
Lemma~\ref{lem:sol_z_basic_prop} implies that the number of roots (counting multiplicities) located in the unit disk equals $n-1$ for each equation~\eqref{eq:beta_z_minus} and~\eqref{eq:beta_z_plus} if $n$ is even, and $n-2$ for equation~\eqref{eq:beta_z_minus} and $n$ for equation~\eqref{eq:beta_z_plus} provided that $n$ is odd. So the total multiplicity of roots of equations~\eqref{eq:beta_z_minus} and~\eqref{eq:beta_z_plus} together equals $2n-2$ regardless the parity of $n$.

Not all of these solutions, however, correspond to an eigenvalue of~$T_{\beta,n}$ for $\beta=\ii h$ and $h>0$. 

\begin{prop}\label{prop:spec_T_char}
 Let $h>0$, and $n\geq2$. Then 
 \[
 \lambda\in\sigma_{d}(T_{\ii h,n}) \quad \Longleftrightarrow \quad  \lambda=\ii h +z+z^{-1}, 
 \]
 for $z\in\C$, $|z|<1$, $\Im z>0$, which is either a solution of~\eqref{eq:beta_z_minus}
 or~\eqref{eq:beta_z_plus}, with $\beta=\ii h$, satisfying the constraint $|z^{n+1}-1|<|z^{n}-z|$ or $|z^{n+1}+1|<|z^{n}+z|$, respectively.
\end{prop}

\begin{proof}
First, note that the Joukowsky transform maps the upper/lower half of the unit disk onto the lower/upper half-plane, i.e., if $0<|z|<1$ and $\Im z\gtrless0$, then $\Im(z+z^{-1})\lessgtr0$. This implies that, among the solutions of~\eqref{eq:beta_z_minus} and~\eqref{eq:beta_z_plus} inside the unit disk, only those with positive imaginary part are of interest. Indeed, if $z$ is a solution of~\eqref{eq:beta_z_minus} or~\eqref{eq:beta_z_plus} with $\Im z<0$, then by~\eqref{eq:bet_z_k}, the equation for the eigenvalue reads $\lambda=z+z^{-1}+\ii h$. But then
\[
 \Im\lambda=h+\Im(z+z^{-1})>h
\]
which is in contradiction with Lemma~\ref{lem:rough_local_evls}. If $\Im z=0$ and $z\notin\{0,\pm1\}$, then $|\Re \lambda|>2$ which is again impossible by Lemma~\ref{lem:rough_local_evls}.

Yet another restriction to solutions of~\eqref{eq:beta_z_minus} and~\eqref{eq:beta_z_plus} has to be imposed. It comes from the necessary requirement $|k|<1$, where $k$ is given by the respective formula from~\eqref{eq:k_rel_z_plusminus} depending on whether $z$ is a solution of~\eqref{eq:beta_z_minus} or~\eqref{eq:beta_z_plus}. On the other hand, if $z$ is a solution of~\eqref{eq:beta_z_minus} or~\eqref{eq:beta_z_plus}, $|z|<1$, $\Im z>0$, and $|k|<1$, then $\lambda=z+z^{-1}+\ii h$ is an eigenvalue of $T_{\ii h,n}$.

Finally, it is straightforward to check that the solutions of \eqref{eq:beta_z_minus} and~\eqref{eq:beta_z_plus} located on the unit circle %, see Lemma~\ref{lem:sol_z_basic_prop}, 
 do not give rise to an eigenvalue. Indeed, 
%for any admissible solution, one finds from
since these solutions satisfy $z\in\{\pm1\}$ (depending on the parity of $n$, see Lemma~\ref{lem:sol_z_basic_prop}), taking the respective limit $z\to\pm 1$ in~\eqref{eq:k_rel_z_plusminus} and using L'Hospital's rule, one finds that
\[
 |k|=\frac{n+1}{n-1}>1.
\]
\end{proof}

A numerical illustration of Proposition~\ref{prop:spec_T_char} is shown in Figure~\ref{fig:num_illust}.

\begin{figure}[htb!]
    \centering
    \begin{subfigure}[b]{0.99\textwidth}
        \includegraphics[width=\textwidth]{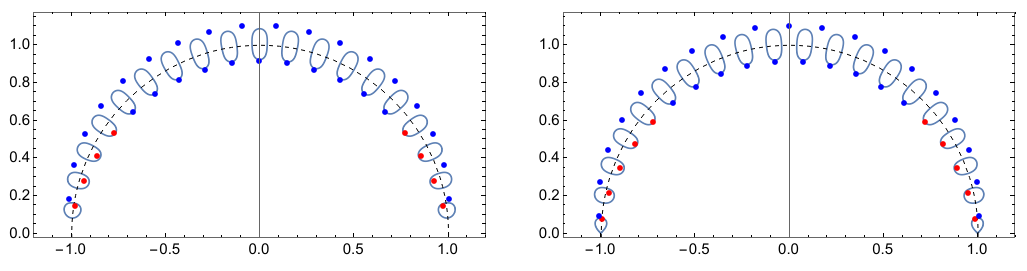}
    \end{subfigure}
    \vskip4pt
    \begin{subfigure}[b]{0.92\textwidth}
			\begin{flushright}    
           \includegraphics[width=\textwidth]{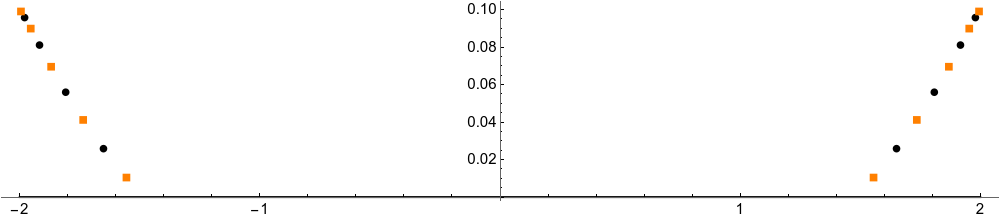}
            \end{flushright}
    \end{subfigure}
       \caption{For parameters $\beta=\ii/10$ and $n=39$, the plots on top show the solutions of~\eqref{eq:beta_z_minus} (left) and \eqref{eq:beta_z_plus} (right) in the upper $z$-plane. The small oval-shaped regions are given by the inequalities $|z^{n+1}-1|<|z^{n}-z|$ (left) and $|z^{n+1}+1|<|z^{n}+z|$ (right), see Proposition~\ref{prop:spec_T_char}. Red dots indicate solutions that are located inside the regions and hence give rise to eigenvalues of $T_{\beta,n}$. Blue dots stay outside the regions. The eigenvalues $\lambda=\ii h+z+z^{-1}$ of $T_{\beta,n}$ are visualized on bottom. Black balls indicate eigenvalues determined by respective solutions of~\eqref{eq:beta_z_minus} and orange squares indicate those given by~\eqref{eq:beta_z_plus}.}
   \label{fig:num_illust}
\end{figure}

\subsection{On the conjecture and the open problem of Hansmann and Katriel}

In this subsection, we let $\beta$ to be purely imaginary and $n$-dependent. Namely $\beta=\beta_{n}:=\ii n^{-2/3}$, which means that
\[
 b_{k}:=\begin{cases}
 		\ii n^{-2/3},& \quad \mbox{ for } k\in\{1,2\dots,n\},\\
 		0,& \quad \mbox{ for } k\in\Z\setminus\{1,2\dots,n\},
 		\end{cases}
\]
and we consider the sequence of discrete Schr{\" o}dinger operators $T_{n}:=T_{\beta_{n},n}$. Note that, as $n\to\infty$, the support of the potential sequence~$b$ is growing while its magnitude $|\beta_{n}|$ tends to zero. The $\ell^{p}$-norm of $b$ is
\[
 \|b\|_{\ell^{p}}=n^{\frac{1}{p}-\frac{2}{3}}.
\]

By means of this particular choice of a sequence of discrete Schr{\" o}dinger operators, we establish Theorems~\ref{thm:first} and~\ref{thm:second}. First, we focus on the statement of Theorem~\ref{thm:first} which follows readily from the following proposition.

\begin{prop}\label{prop:h-k_conj}
 For $\omega<p$, one has
 \[
 \lim_{n\to\infty}n^{\frac{2p-3}{3}}\sum_{\lambda\in\sigma_{d}(T_{n})}\!\!\left(\dist(\lambda,[-2,2])\right)^{\omega}=\infty.
 \]
\end{prop}

\begin{proof}
 We make use of the characterization of discrete eigenvalues of~$T_{n}$ given in Proposition~\ref{prop:spec_T_char} via solutions of the equations \eqref{eq:beta_z_minus} and~\eqref{eq:beta_z_plus}. In fact, for the purpose of this proof, it is sufficient to focus on solutions of~\eqref{eq:beta_z_minus} located in a particular subregion of the unit disk. 
 
 More concretely, we seek solutions $z=re^{\ii\phi}$ of the equation~\eqref{eq:beta_z_minus}, with $\beta=\ii n^{-2/3}$, in the compact region determined by the restrictions
 \begin{equation}
 \frac{\pi}{4}\leq\phi\leq\frac{3\pi}{4} \quad\mbox{ and }\quad 1-\frac{1}{\sqrt{n}}\leq r \leq1-c\frac{\log n}{n},
 \label{eq:phi_r_restrict}
 \end{equation}
 where $c\in(1/2,2/3)$ is arbitrary but fixed. Actually, the choice for the range of~$\phi$ is taken for the sake of concreteness, any closed subinterval of $(0,\pi)$ could be taken.
 It follows from~\eqref{eq:phi_r_restrict} that
 \[
  r^{n}\leq\left(1-c\frac{\log n}{n}\right)^{n}=n^{-c}\left(1+O\left(\frac{\log^{2} n}{n}\right)\right), \quad n\to\infty.
 \]
 In particular, we may write
 \begin{equation}
  r^{n}=O\left(n^{-c}\right), \quad n\to\infty.
  \label{eq:r_n_asympt}
 \end{equation}
 On the other hand, again according to~\eqref{eq:phi_r_restrict}, one has
 \begin{equation}
  r=1+O\left(\frac{1}{\sqrt{n}}\right), \quad n\to\infty.
  \label{eq:r_asympt}
 \end{equation}
 
 For $z=re^{\ii\phi}$ and $\beta=\ii n^{-2/3}$, the equation~\eqref{eq:beta_z_minus} reads
 \[
  \ii n^{-2/3}\left(1-r^{n+1}e^{\ii(n+1)\phi}\right)\left(1-r^{n-1}e^{\ii(n-1)\phi}\right)=r^{n}e^{\ii n\phi}\left(re^{\ii\phi}-r^{-1}e^{-\ii\phi}\right)^{2}.
 \]
 Using~\eqref{eq:r_n_asympt} and~\eqref{eq:r_asympt}, one gets the asymptotic equality
 \[
  \ii n^{-2/3}\left(1+O\left(n^{-c}\right)\right)=-4\left(\sin^{2}\phi\right)r^{n}e^{\ii n\phi}\left(1+O\left(\frac{1}{\sqrt{n}}\right)\right),
 \]
 for $n\to\infty$. Note that the error terms actually hold uniformly in~$\phi$. Notice also that the term $\sin^{2}\phi$ stays bounded away from zero by our assumptions. More precisely, $\sin^{2}\phi\in[1/2,1]$ which follows from the restriction on~$\phi$ from~\eqref{eq:phi_r_restrict}.
 Finally, taking also into account that $c\in(1/2,2/3)$, we arrive at the asymptotic formula
 \begin{equation}
  4\ii\left(\sin^{2}\phi\right)n^{2/3}r^{n}e^{\ii n\phi}=1+O\left(\frac{1}{\sqrt{n}}\right), \quad n\to\infty.
  \label{eq:asymt_eq_solut}
 \end{equation}
 
 Taking the arguments in~\eqref{eq:asymt_eq_solut}, one observes that the argument of a solution has to fulfill
 \begin{equation}
  \phi=\phi_{j}=\frac{\pi(4j-1)}{2n}+O\left(\frac{1}{n^{3/2}}\right), \quad n\to\infty,
 \label{eq:phi_j_asympt}
 \end{equation}
 for $j\in\Z$. Bearing in mind the supposed restriction on~$\phi$ from~\eqref{eq:phi_r_restrict}, we choose the range for the index $j$ to be
 \begin{equation}
  \frac{n+2}{8}\leq j \leq \frac{3n+2}{8}.
 \label{eq:range_j}
 \end{equation}

 Taking modulus in~\eqref{eq:asymt_eq_solut}, one obtains for the modulus of a solution 
 \[
  r_{j}=\left[\frac{n^{-2/3}}{4\sin^{2}\phi_{j}}\left(1+O\left(\frac{1}{\sqrt{n}}\right)\right)\right]^{1/n}, \quad n\to\infty.
 \]
 Since $4\sin^{2}\phi_{j}\in[2,4]$ for any $j$ satisfying~\eqref{eq:range_j}, 
 \[
 \left(4\sin^{2}\phi_{j}\right)^{-1/n}=1+O\left(\frac{1}{n}\right),\quad n\to\infty,
 \]
 uniformly for all $j$ admissible. Then a straightforward calculation yields
 \begin{equation}
  r_{j}=1-\frac{2}{3}\frac{\log n}{n}+O\left(\frac{1}{n}\right), \quad n\to\infty,
 \label{eq:r_j_asympt}
 \end{equation}
 uniformly in $j$. Note that the found $r_{j}$ fulfills the restriction~\eqref{eq:phi_r_restrict} for $n$ sufficiently large. In total, we see that there are asymptotically $n/4$ solutions $z_{j}=r_{j}e^{\ii\phi_{j}}$ of~\eqref{eq:beta_z_minus} within the region~\eqref{eq:phi_r_restrict}, with $j$ as in~\eqref{eq:range_j}, and asymptotic expansions for their arguments and moduli are given by equations~\eqref{eq:phi_j_asympt} and~\eqref{eq:r_j_asympt}.

 Next, we show that the found solutions $z_{j}$ give rise to eigenvalues of~$T_{n}$, if $n$ is sufficiently large. To this end, according to Proposition~\ref{prop:spec_T_char}, one has to check that $|k_{j}|<1$, where
 \[
  k_{j}:=\frac{1-z_{j}^{n+1}}{z_{j}-z_{j}^{n}}.
 \]
 The verification proceeds as follows. Taking~\eqref{eq:r_n_asympt} into account, we obtain
 \[
  \frac{1-z_{j}^{n+1}}{1-z_{j}^{n-1}}=\left(1-z_{j}^{n+1}\right)\left(1+z_{j}^{n-1}+O\left(n^{-2c}\right)\right)=1-\left(z_{j}-z_{j}^{-1}\right)z_{j}^{n}+O\left(n^{-2c}\right),
 \]
for $n\to\infty$. Using further that $z_{j}$ is a solution of~\eqref{eq:beta_z_minus}, together with formulas~\eqref{eq:r_n_asympt} and~\eqref{eq:r_j_asympt}, we get
\[
 \left(z_{j}-z_{j}^{-1}\right)z_{j}^{n}=\ii n^{-2/3}\frac{(1-z_{j}^{n+1})(1-z_{j}^{n-1})}{z_{j}-z_{j}^{-1}}=\frac{n^{-2/3}}{2\sin\phi_{j}}\left(1+O\left(n^{-c}\right)\right), \quad n\to\infty.
\]
In total, we have
\[
 \frac{1-z_{j}^{n+1}}{1-z_{j}^{n-1}}=1-\frac{n^{-2/3}}{2\sin\phi_{j}}+O\left(n^{-2c}\right), \quad n\to\infty,
\]
where $c\in(1/2,2/3)$. Hence, using~\eqref{eq:r_j_asympt} once more, we arrive at the expansion
\[
 k_{j}=e^{-\ii\phi_{j}}\left(1-\frac{n^{-2/3}}{2\sin\phi_{j}}+O\left(\frac{\log n}{n}\right)\right), \quad n\to\infty.
\]
Since $\sin\phi_{j}>0$, we observe that $|k_{j}|<1$ for $n$ sufficiently large. Moreover, for the respective eigenvalue, we obtain 
\begin{equation}
\lambda_{j}=k_{j}+k_{j}^{-1}=2\cos\phi_{j}+\ii n^{-2/3}+O\left(\frac{\log n}{n}\right), \quad n\to\infty.
\label{eq:lam_j_asympt}
\end{equation}

Consequently, it follows from~\eqref{eq:lam_j_asympt} that
\begin{equation}
 \dist(\lambda_{j},[-2,2])=\Im\lambda_{j}=n^{-2/3}+O\left(\frac{\log n}{n}\right), \quad n\to\infty. 
\label{eq:dist_asympt}
\end{equation}
Recall that the indices $j$ are restricted as in~\eqref{eq:range_j} and so the number of eigenvalues of $T_{n}$, which the analysis is restricted to, is asymptotically $n/4$ for large $n$. Thus, we may estimate
\[
 \sum_{\lambda\in\sigma_{d}(T_{n})}\!\!\left(\dist(\lambda,[-2,2])\right)^{\omega}\geq\frac{n}{4}\left(n^{-2/3}+O\left(\frac{\log n}{n}\right)\right)^{\omega}=\frac{n^{1-2\omega/3}}{4}\left(1+O\left(\frac{\log n}{n^{1/3}}\right)\right),
\]
for $n\to\infty$. Therefore
\[
 n^{\frac{2p-3}{3}}\sum_{\lambda\in\sigma_{d}(T_{n})}\!\!\left(\dist(\lambda,[-2,2])\right)^{\omega}\geq\frac{n^{2(p-\omega)/3}}{4}\left(1+O\left(\frac{\log n}{n^{1/3}}\right)\right), \quad n\to\infty,
\]
from which the statement follows.
\end{proof}

The chosen sequence of operators~$T_{n}$ also exhibits properties that imply Theorem~\ref{thm:second}. These properties are established in the next result. 

\begin{prop}
 For any $\sigma\geq 1/2$ and $p\geq1$, one has
 \[
 \lim_{n\to\infty}n^{\frac{2p-3}{3}}\sum_{\lambda\in\sigma_{d}(T_{n})}\frac{\dist\left(\lambda,[-2,2]\right)^{p}}{|\lambda^{2}-4|^{\sigma}}=\infty.
\]
\end{prop}

\begin{proof}
The first part of the proof is a moderate modification of the approach applied in the proof of Proposition~\ref{prop:h-k_conj}. The essential difference is that one has to take into account the eigenvalues of $T_{n}$ occurring in the neighborhoods of the endpoints $\pm2$ of the essential spectrum $[-2,2]$. These eigenvalues were excluded from the previous analysis by restricting the range of $\phi$ in~\eqref{eq:phi_r_restrict}. At this point, we need to allow $\phi$ to approach $0$ arbitrarily close. Therefore we extend the range for the angle $\phi$ supposing
\begin{equation}
 \epsilon\pi\leq\phi\leq(1-\epsilon)\pi,
\label{eq:phi_restr_extended}
\end{equation}
for arbitrary but fixed $0<\epsilon<1/2$. Then $\sin\phi$ still remains bounded away from zero and the same computation as in the proof of Proposition~\ref{prop:h-k_conj} yields that, for $n$ large, there are asymptotically $n(1-2\epsilon)/2$ eigenvalues $\lambda_{j}$ of $T_{n}$ with the asymptotic behavior~\eqref{eq:lam_j_asympt}. The adapted range for indices~$j$ is given by inequalities 
\[
 \epsilon\pi\leq\frac{\pi(4j-1)}{2n}\leq(1-\epsilon)\pi.
\]
Then the argument $\phi_{j}$ satisfies~\eqref{eq:phi_restr_extended}, as $n\to\infty$, see~\eqref{eq:phi_j_asympt}. It means the range for~$j$ now reads
\begin{equation}
 \frac{2n\epsilon+1}{4}\leq j \leq \frac{2n(1-\epsilon)+1}{4}.
\label{eq:range_j_extended}
\end{equation}

The asymptotic formula~\eqref{eq:dist_asympt} remains true in the same form. Consequently, for all $n$ sufficiently large and $j$ satisfying~\eqref{eq:range_j_extended}, one has
\begin{equation}
 \dist(\lambda_{j},[-2,2])\geq\frac{1}{2}n^{-2/3}.
 \label{eq:dist_ineq}
\end{equation}
In addition, for $\sigma\geq1/2$, one gets
\begin{equation}
 \liminf_{n\to\infty}\frac{1}{n}\sum_{\lambda\in\sigma(T_{n})}\frac{1}{|\lambda^{2}-4|^{\sigma}}
 \geq\liminf_{n\to\infty}\frac{1}{n}\sum_{j\mbox{ as in }\eqref{eq:range_j_extended}}\frac{1}{|\lambda_{j}^{2}-4|^{\sigma}}=\frac{1}{2^{2\sigma+1}\pi}\int_{\epsilon\pi}^{(1-\epsilon)\pi}\frac{\dd x}{(1-\cos^{2} x)^{\sigma}},
 \label{eq:aux_ineq1}
\end{equation}
where~\eqref{eq:lam_j_asympt} has been used. Further, we estimate the integral
\begin{equation}
 \int_{\epsilon\pi}^{(1-\epsilon)\pi}\frac{\dd x}{(1-\cos^{2} x)^{\sigma}}=2\int_{\epsilon\pi}^{\pi/2}\frac{\dd x}{\sin^{2\sigma}x}\geq2\int_{\epsilon\pi}^{1}\frac{\dd x}{x^{2\sigma}}=
 \begin{cases}
 \frac{2}{2\sigma-1}\left((\pi\epsilon)^{1-2\sigma}-1\right), &\mbox{ if } \sigma>\frac{1}{2},\\[4pt]
 -2\log(\pi\epsilon), &\mbox{ if } \sigma=\frac{1}{2}.
 \end{cases}
 \label{eq:aux_ineq2}
\end{equation}
A combination of~\eqref{eq:aux_ineq1} and~\eqref{eq:aux_ineq2} implies that
\begin{equation}
 \liminf_{n\to\infty}\frac{1}{n}\sum_{\lambda\in\sigma(T_{n})}\frac{1}{|\lambda^{2}-4|^{\sigma}}\geq C_{\sigma}(\epsilon),
 \label{eq:r-sum_ineq}
\end{equation}
where
\[
 C_{\sigma}(\epsilon):=\frac{1}{2^{2\sigma}\pi}\times\begin{cases}
 \frac{1}{2\sigma-1}\left((\pi\epsilon)^{1-2\sigma}-1\right), &\mbox{ if } \sigma>\frac{1}{2},\\[4pt]
 -\log(\pi\epsilon), &\mbox{ if } \sigma=\frac{1}{2}.
 \end{cases}
\]
Clearly, for any $\sigma\geq1/2$,
\begin{equation}
 \lim_{\epsilon\to0+}C_{\sigma}(\epsilon)=\infty.
 \label{eq:C_eps_lim}
\end{equation}

Finally, one makes use of~\eqref{eq:dist_ineq} together with~\eqref{eq:r-sum_ineq} to obtain the lower bound
\[
 \liminf_{n\to\infty}n^{\frac{2p-3}{3}}\sum_{\lambda\in\sigma_{d}(T_{n})}\frac{\dist\left(\lambda,[-2,2]\right)^{p}}{|\lambda^{2}-4|^{\sigma}}\geq 2^{-p}\,C_{\sigma}(\epsilon).
\]
Bearing in mind~\eqref{eq:C_eps_lim} and taking the limit $\epsilon\to0+$ in the above inequality, one proves the statement.
\end{proof}

\section{Schr{\" o}dinger operators in dimension one}\label{sec:schrodinger}

The following theorem is a continuous analogue to Theorem~\ref{thm:second} and particularly
yields the negative answer to the open problem from~\cite{dem-han-kat_ieop13} for one-dimensional Schr{\" o}dinger operators $H=\frac{{\rm d}^2}{{\rm d}x^2}+V$ with complex-valued potentials $V\in L^{p}(\R)$ and $p\geq1$. 

\begin{thm}\label{thm:third}
For any $p\geq1$ and $\sigma\geq1/2$, one has
\[
\sup_{0\neq V\in L^{p}(\R)}\frac{1}{\|V\|_{L^{p}}^{p}}\sum_{\lambda\in\sigma_{d}(H)}\frac{\left(\dist\left(\lambda,[0,\infty)\right)\right)^{p}}{|\lambda|^{\sigma}}=\infty.
\]
\end{thm}

The proof of Theorem~\ref{thm:third} follows from the following asymptotic analysis of discrete eigenvalues of the Schr{\" o}dinger operator with rectangular barrier potential and complex coupling constant.

\subsection{One-dimensional Schr{\" o}dinger operator with rectangular barrier potential and complex coupling}

Our strategy proceeds similarly as in the discrete settings. However, a scaling of the variable allows us to restrict the analysis to an even simpler family of Schr{\" o}dinger operators with a rectangular potential of a fixed support. Concretely, we study the one-parameter family of Schr{\" o}dinger operators $H_{h}:=H_{0}+V_{h}$ acting on $L^{2}(\R)$ with the potential
\begin{equation}
V_{h}(x):=\ii h\chi_{[-1,1]}(x), \quad x\in\R,
\label{eq:def_V_h}
\end{equation}
where $h>0$ and $\chi_{[-1,1]}$ is the indicator function of the interval $[-1,1]$. More concretely, the asymptotic behavior of the discrete eigenvalues of $H_{h}$, for $h\to\infty$, located in a subset of the complex plane is of our primary interest and, in the end,  yields a proof for Theorem~\ref{thm:third}.

The general analysis of the discrete eigenvalues of $H_{h}$ proceeds in a standard fashion by solving the eigenvalue equation $H_{h}\psi=\lambda\psi$ separately on $(-1,1)$ and $\R\setminus[-1,1]$ and choosing $\psi$, as well as its derivative, to be continuous at $\pm1$ . As a result, one finds that $\lambda=k^{2}$ is an eigenvalue of $H_{h}$, if there exist $\mu\in\C\setminus(2\Z+1)\frac{\pi}{2}$ satisfying the equations
\begin{equation}
 k^{2}=\mu^{2}+\ii h \quad\mbox{ and }\quad k=\ii\mu\tan\mu
\label{eq:k_mu_rel}
\end{equation}
together with the restriction
\begin{equation}
\Re\left(\mu\tan\mu\right)>0.
\label{eq:mu_cond}
\end{equation}
The last inequality means nothing but $\Im k>0$. Then the eigenvector of $H_{h}$ corresponding to the eigenvalue $\lambda=k^{2}$ can be chosen as
\[
\psi(x)=\begin{cases}
\cos(\mu)e^{\ii k|x|}, & \mbox{ if } |x|>1,\\
e^{\ii k}\cos(\mu x), & \mbox{ if } |x|\leq1.
\end{cases}
\]
The equations in~\eqref{eq:k_mu_rel} provide us with the characteristic equation
\begin{equation}
\mu^{2}+\ii h \cos^{2}\mu=0,
\label{eq:char_eq_mu}
\end{equation}
whose solutions $\mu$ are restricted by~\eqref{eq:mu_cond}. 

Finally, one can show, similarly as in Lemma~\ref{lem:rough_local_evls}, that the discrete spectrum of $H_{h}$ has to be located in the strip $[0,\infty)+\ii(0,h]$, for $h>0$. A numerical illustration of the discrete spectrum of $H_{h}$, for $h=2500$, is shown in Figure~\ref{fig:evls-cont}.

\begin{figure}[htb!]
    \centering
        \includegraphics[width=0.9\textwidth]{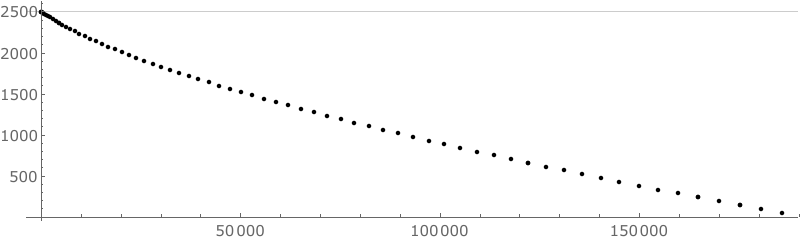}
       \caption{The points represent eigenvalues $\lambda=\mu^{2}+\ii h$ of $H_{h}$, for $h=2500$, where $\mu$ are numerically found solutions of the equation~\eqref{eq:char_eq_mu} satisfying condition~\eqref{eq:mu_cond}.}
   \label{fig:evls-cont}
\end{figure}

\subsection{On the problem of Demuth, Hansmann, and Katriel}

Analogously to the discrete case, we first consider the Schr{\" o}dinger operator $\tilde{H}_{h}$ on $L^{2}(\R)$ defined by
\[
\tilde{H}_{h}:=-\frac{\dd^{2}}{\dd x^{2}}+\tilde{V}_{h},
\]
where
\[
 \tilde{V}_{h}(x):=\frac{\ii}{h}\chi_{[-h,h]}(x)
\]
and $h>0$. The operator $U_{h}$ defined on $L^{2}(\R)$ by 
\[
U_{h}f(x):=h^{1/2}f(hx), \quad x\in\R,
\]
is an isomorphism on $L^{2}(\R)$. Moreover, one has
\[
h^{2}\tilde{H}_{h}=U_{h}^{-1}H_{h}U_{h},
\]
where $H_{h}$ is the Schr{\" o}dinger operator with potential~\eqref{eq:def_V_h}. 
It follows that
\[
\lambda\in\sigma(H_{h}) \quad\Leftrightarrow\quad \frac{\lambda}{h^{2}}\in\sigma(\tilde{H}_{h}).
\]
Noticing also that
\[
\|\tilde{V}_{h}\|_{L^{p}}^{p}=2h^{1-p}
\]
and
\[
 \dist(\lambda,[0,\infty))=\Im\lambda, \; \mbox{ for } \lambda\in\sigma_{d}(\tilde{H}_{h}),
\]
one obtains
\begin{equation}
\frac{1}{\|\tilde{V}_{h}\|_{L^{p}}^{p}}\sum_{\lambda\in\sigma_{d}(\tilde{H}_{h})}\frac{\left(\dist(\lambda,[0,\infty))\right)^{p}}{|\lambda|^{\sigma}}=\frac{1}{2}h^{2\sigma-p-1}\sum_{\lambda\in\sigma_{d}(H_{h})}\frac{(\Im\lambda)^{p}}{|\lambda|^{\sigma}}.
\label{eq:lieb-thirring-expr-tilde-vs-non-tilde}
\end{equation}
Now equality~\eqref{eq:lieb-thirring-expr-tilde-vs-non-tilde} together with the following statement implies Theorem~\ref{thm:third}.

\begin{prop}
For any $p\geq1$ and $\sigma\geq1/2$, it holds that
\[
\lim_{h\to\infty}h^{2\sigma-p-1}\sum_{\lambda\in\sigma_{d}(H_{h})}\frac{(\Im\lambda)^{p}}{|\lambda|^{\sigma}}=\infty.
\]
\end{prop}

\begin{proof}
 First note that the function on the left-hand side of~\eqref{eq:char_eq_mu} is even in $\mu$. Moreover,  equation \eqref{eq:char_eq_mu} does not possess any purely imaginary solutions. Thus, we can restrict the analysis of solutions  of~\eqref{eq:char_eq_mu} to the half-plane given by $\Re\mu<0$.
 
 In the proof, we are interested in those solutions $\mu$ of~\eqref{eq:char_eq_mu} that are located in the set determined by the inequalities
 \begin{equation}
 \alpha\log h\leq\Im\mu\leq\beta\log h \quad \mbox{ and } \quad \Re\mu\leq-h^{\gamma}|\Im\mu|,
 \label{eq:mu_ineq}
 \end{equation}
where $\alpha,\beta,\gamma>0$ are ($h$-independent) constants such that $\gamma<2\alpha<2\beta<1$. Such a set does indeed contain solutions of~\eqref{eq:char_eq_mu}, if $h$ is large enough. In fact, we will show that the number of solutions is increasing as $h\to\infty$ and their asymptotic behavior implies the claim.

It follows from~\eqref{eq:mu_ineq} that 
 \[
 \cos\mu=\frac{e^{-\ii\mu}}{2}\left(1+O\left(h^{-2\alpha}\right)\right) \quad\mbox{ and }\quad \mu=\left(\Re\mu\right)\left(1+O\left(h^{-\gamma}\right)\right),
 \]
 as $h\to\infty$. Note that both error terms in the above asymptotic formulas are independent of $\mu$ and hence the asymptotic expansions are uniform in~$\mu$ from~\eqref{eq:mu_ineq}. Further, since $\gamma<2\alpha$, one has
 \begin{equation}
 \frac{\mu}{\cos{\mu}}=2\left(\Re\mu\right)e^{\ii\mu}\left(1+O\left(h^{-\gamma}\right)\right), \quad h\to\infty.
 \label{eq:mu_cosmu_aux}
 \end{equation}
 
 Rather than~\eqref{eq:char_eq_mu}, we actually focus on solutions of the equation
 \begin{equation}
  \mu+e^{-\ii\pi/4}\sqrt{h}\cos\mu=0,
 \label{eq:char_eq_root}
 \end{equation}
 which are clearly also solutions of~\eqref{eq:char_eq_mu}. By combining \eqref{eq:mu_cosmu_aux} and \eqref{eq:char_eq_root}, one arrives at the equation
 \begin{equation}
 2\left(\Re\mu\right)e^{\ii\mu}=-\sqrt{h}e^{7\ii\pi/4}\left(1+O\left(h^{-\gamma}\right)\right),
 \label{eq:char_eq_root_leading_terms}
 \end{equation}
 for $h\to\infty$. 
 
 Next, we will need the following general observation: for given $r>0$ and $\varphi\in[0,2\pi)$, all solutions $\mu\in\C$, with $\Re\mu<0$, of the equation
 \[
 2\left(\Re\mu\right)e^{\ii\mu}=-re^{\ii\varphi}
 \]
 are
 \[
 \mu_{j}=\varphi-2j\pi+\ii\log\left(\frac{4j\pi-2\varphi}{r}\right), \quad j\in\N.
 \]
  Applying this observation to~\eqref{eq:char_eq_root_leading_terms} with
 \[ 
 r=\sqrt{h}\left(1+O\left(h^{-\gamma}\right)\right) \quad \mbox{ and } \quad  \varphi=\frac{7\pi}{4}\left(1+O\left(h^{-\gamma}\right)\right),
 \]
 one gets asymptotic formulas for solutions of~\eqref{eq:char_eq_root} in the form
 \begin{equation}
 \mu_{j}=\frac{\pi}{4}(7-8j)+\ii\log\left(\frac{\pi(8j-7)}{2\sqrt{h}}\right)+O\left(h^{-\gamma}\right), \quad h\to\infty,
 \label{eq:mu_j_asympt}
 \end{equation}
 provided that the indices $j\in\N$ are taken such that the $\mu_{j}$ satisfy the restrictions from~\eqref{eq:mu_ineq}. The first restriction from~\eqref{eq:mu_ineq} imposes $h^{\alpha}\leq\exp(\Im\mu_{j})\leq h^{\beta}$, which means
 \[
 \frac{h^{\alpha+1/2}}{4\pi}\left(1+O\left(h^{-\gamma}\right)\right)+\frac{7}{8}\leq j \leq  \frac{h^{\beta+1/2}}{4\pi}\left(1+O\left(h^{-\gamma}\right)\right)+\frac{7}{8}.
 \]
 Using that the Landau symbols above do not depend on $j$, we can restrict the range for $j$ even more. In fact, due to the freedom of choice of constants $\alpha$ and $\beta$ satisfying $\gamma<2\alpha<2\beta<1$, we can simply suppose
 \begin{equation}
	h^{\alpha+1/2}\leq j\leq h^{\beta+1/2},
	\label{eq:range_j_h}
 \end{equation}
 for $h$ sufficiently large, without loss of generality. Concerning the second restriction from~\eqref{eq:mu_ineq}, it is straightforward to check that it is automatically satisfied for $\mu=\mu_{j}$, if $h$ is sufficiently large.
 
 Further, we show that the found solutions $\mu_{j}$, with $j$ as in~\eqref{eq:range_j_h}, give rise to eigenvalues of $H_{h}$ for $h$ large enough. To do so, one has to verify condition~\eqref{eq:mu_cond}, which is equivalent to 
 \[
(\Im\mu)\sinh(2\Im\mu)<(\Re\mu)\sin(2\Re\mu),
 \]
  for $\mu=\mu_{j}$. To this end, we use the asymptotic expansions
 \[
  \sin(2\Re\mu_{j})=-1+O\left(h^{-\gamma}\right), \quad \Re\mu_{j}=-\frac{1}{2}\sqrt{h}e^{\Im\mu_{j}}\left(1+O\left(h^{-\gamma}\right)\right), \quad  h\to\infty,
 \]
 and the inequalities
 \[
  \sinh(2\Im\mu_{j})<\frac{1}{2}e^{2\Im\mu_{j}}, \quad \Im\mu_{j}\leq \beta\log h,
 \]
  which are consequences of~\eqref{eq:mu_j_asympt} and~\eqref{eq:mu_ineq}. Then, since $\beta<1/2$, we have
  \[
   (\Im\mu_{j})\sinh(2\Im\mu_{j})<\frac{\beta}{2}\log(h)h^{\beta}e^{\Im\mu_{j}}<\frac{\beta}{2}\sqrt{h}e^{\Im\mu_{j}}<(\Re\mu_{j})\sin(2\Re\mu_{j}),
  \]
  provided $h$ to be sufficiently large.
  
  According to~\eqref{eq:k_mu_rel}, the eigenvalue $\lambda_{j}$ corresponding to the solution~$\mu_{j}$ is given by  $\lambda_{j}=\ii h +\mu_{j}^{2}$. It follows from~\eqref{eq:mu_j_asympt} and~\eqref{eq:range_j_h} that
  \[
  \Im\lambda_{j}=h+\Im(\mu_{j}^{2})=h+2\Re\mu_{j}\Im\mu_{j}=h+O\left(h^{\beta+1/2}\log h\right)=h\left(1+O\left(h^{\beta-1/2}\log h\right)\right)\!,
  \]
  as $h\to\infty$. In particular, we may conclude that there exists $h_{0}>0$ such that, for $h>h_{0}$ and $j$ satisfying~\eqref{eq:range_j_h}, we have the estimate 
  \begin{equation}
   \Im\lambda_{j}>\frac{h}{2}.
   \label{eq:im_lam_bound}
  \end{equation}
  Similarly, one computes that
  \begin{equation}
  |\lambda_{j}|=|\mu_{j}|^{2}\left|1+\frac{\ii h}{\mu_{j}^{2}}\right|=\left(\Re\mu_{j}\right)^{2}\left(1+O\left(h^{-2\gamma}\right)\right), \quad h\to\infty,
  \label{eq:abs_lam_j_aux}
  \end{equation}
  where we have used the assumption $\gamma<2\alpha$ together with the asymptotic formulas
  \[
  \frac{1}{\mu_{j}}=O\left(h^{-\alpha-1/2}\right) \quad \mbox{ and } \quad |\mu_{j}|^{2}=\left(\Re\mu_{j}\right)^{2}\left(1+O\left(h^{-2\gamma}\right)\right), \quad h\to\infty.
  \]
  It follows again from~\eqref{eq:mu_j_asympt} and~\eqref{eq:range_j_h} that $0<-\Re\mu_{j}\leq2\pi  j$, for all $h$ sufficiently large. Using~\eqref{eq:abs_lam_j_aux}, we may claim without loss of generality that, for $h>h_{0}$ and $j$ within the range~\eqref{eq:range_j_h}, we have the estimate
  \begin{equation}
  |\lambda_{j}|^{1/2}\leq2\pi j.
  \label{eq:abs_lam_bound}
  \end{equation}
  
  In total, for $\sigma\geq1/2$, estimates~\eqref{eq:im_lam_bound} and~\eqref{eq:abs_lam_bound} yield the lower bound
  \[
  \sum_{\lambda\in\sigma_{d}(H_{h})}\frac{\left(\Im\lambda\right)^{p}}{|\lambda|^{\sigma}}\geq
  \frac{h^{p}}{2^{2\sigma+p}\pi^{2\sigma}}\sum_{h^{\alpha+1/2}\leq j\leq h^{\beta+1/2}}\frac{1}{j^{2\sigma}},
  \]
  for all $h>h_{0}$. If we further apply the estimate
  \[
   \sum_{u\leq j \leq v}\frac{1}{j^{2\sigma}}\geq \int_{u}^{v}\frac{\dd x}{x^{2\sigma}}-\frac{1}{(u-1)^{2\sigma}}-\frac{1}{v^{2\sigma}},
  \]
  which holds true for any $1<u<v$, we obtain
  \[
  \sum_{\lambda\in\sigma_{d}(H_{h})}\frac{\left(\Im\lambda\right)^{p}}{|\lambda|^{\sigma}}\geq
  \begin{cases}
  	\frac{h^{p}}{2^{p+1}\pi}\left[(\beta-\alpha)\log h +O\left(h^{-\alpha-1/2}\right)\right],& \quad \mbox{ if } \sigma=1/2,\\[4 pt]
  	\frac{h^{p}}{2^{2\sigma+p}\pi^{2\sigma}(2\sigma-1)}\left[h^{(1-2\sigma)(\alpha+1/2)} +O\left(h^{-\delta}\right)\right],& \quad \mbox{ if } \sigma>1/2,\\
  	\end{cases}
  \]
  for $h\to\infty$, where $\delta:=\min\{2\sigma(\alpha+1/2),(2\sigma-1)(\beta+1/2)\}$.
  Finally, we arrive at the lower estimate
  \[
  h^{2\sigma-p-1}\sum_{\lambda\in\sigma_{d}(H_{h})}\frac{\left(\Im\lambda\right)^{p}}{|\lambda|^{\sigma}}\geq
  \begin{cases}
  	\frac{\beta-\alpha}{2^{p+1}\pi}\log(h) \left(1 +o(1)\right),& \quad \mbox{ if } \sigma=1/2,\\[4 pt]	
  	\frac{1}{2^{2\sigma+p}\pi^{2\sigma}(2\sigma-1)}h^{(2\sigma-1)(1/2-\alpha)}\left(1 +o(1)\right),& \quad \mbox{ if } \sigma>1/2,\\
  	\end{cases}
  \]
  as $h\to\infty$, which readily implies the statement.
\end{proof}

\subsection{A comment on the multidimensional case}
\label{subsec:higherdim}

The open problem from~\cite{dem-han-kat_ieop13} concerns arbitrary dimensions~$d\in\N$. At this point, when the solution is found for $d=1$, one can naturally ask whether the approach used in the one-dimensional case could be generalized to find counter-examples in the multidimensional case as well. Clearly, there are many candidates that could be thought of as multidimensional analogues of the Schr{\" o}dinger operators~$H_{h}$ analyzed in this Section~\ref{sec:schrodinger}.
Except the requirement that the multidimensional candidate should coincide with~$H_{h}$ for $d=1$, one should also seek operators whose spectral problem can be reduced to a problem of finding zeros of some well known functions.

One of possible candidates is given by the family of Schr{\" o}dinger operators
\[
H_{h}:=-\Delta+\ii h\chi_{B_{1}(0)}, \quad h>0,
\]
where $\chi_{B_1(0)}$ is the indicator function of the $d$-dimensional unit ball~$B_{1}(0)$ centered at the origin. Since the potential is spherically symmetric, it is natural to use spherical coordinates in the spectral analysis of~$H_{h}$. Then the eigenvalue equation for the radial part of the transformed operator~$H_{h}$ reduces to the Bessel differential equation. The requirement that the eigenfunctions have to be continuously differentiable at the unit sphere provides us with a~characteristic equation expressed in terms of Bessel and Hankel functions of the first kind. However, the necessary asymptotic analysis of the eigenvalues seems to be much more involved than in the particular case of $d=1$.

At this moment, we do not known whether $H_{h}$ can serve as a counter-example for the open problem of Demuth, Hansmann, and Katriel when $d\geq2$. This question should be the subject of future research.

\section*{Acknowledgement}
The research of F.~{\v S}. was supported by the GA{\v C}R grant No. 20-17749X.

\bibliographystyle{acm}

\end{document}